\documentclass[10pt]{amsart}
\usepackage{a4, amsmath, bm, amssymb}
\usepackage{setspace}
\usepackage{subfiles}
\usepackage{cite}
\usepackage{enumerate}
\usepackage{url}

\usepackage{a4,amsmath,amssymb,amscd,verbatim,bbm,graphicx,enumerate,tikz}

\definecolor{crimson}{rgb}{0.86, 0.08, 0.24}
\definecolor{bleudefrance}{rgb}{0.19, 0.5, 0.91}
\usepackage[pdftex,colorlinks,linkcolor=bleudefrance,citecolor=crimson,filecolor=black]{hyperref}

\usepackage{faktor}
\usepackage{dsfont}
\usepackage{amsthm}

\makeatletter
\newtheorem*{rep@theorem}{\rep@title}
\newcommand{\newreptheorem}[2]{%
\newenvironment{rep#1}[1]{%
 \def\rep@title{#2 \ref{##1}}%
 \begin{rep@theorem}}%
 {\end{rep@theorem}}}
\makeatother

\newtheorem{theorem}{Theorem}
\newtheorem{lemma}[theorem]{Lemma}

\newtheorem{proposition}[theorem]{Proposition}
\newtheorem{corollary}[theorem]{Corollary}

\theoremstyle{definition}

\newtheorem{remark}[theorem]{Remark}

\numberwithin{equation}{section}
\numberwithin{theorem}{section}

\def\N{\mathbb{N}}

\def\Z{\mathbb{Z}}
\def\R{\mathbb{R}}

\begin{document}
  
\normalsize

\title[A non-symmetric Kesten criterion]{A non-symmetric Kesten criterion and ratio
limit theorem for random walks on amenable groups}

\author{Rhiannon \textsc{Dougall} and Richard \textsc{Sharp}}

\address{Department of Mathematical Sciences,
Durham University,
Upper Mountjoy,
Durham DH1 3LE}
\email{rhiannon.dougall@durham.ac.uk}

\address{Mathematics Institute, University of Warwick,
Coventry CV4 7AL, UK}
\email{R.J.Sharp@warwick.ac.uk}

\thanks{\copyright 2023. This work is licensed by a CC BY license.}

\begin{abstract}
We consider random walks on countable groups.
A celebrated result of Kesten says that the spectral radius of a symmetric walk
(whose support generates the group as a semigroup) is equal to one if and only if the group is
amenable. We give an analogue of this result for finitely supported walks
which are not symmetric. 
We also conclude a ratio 
limit theorem for amenable groups.
\end{abstract}

\maketitle

\section{Introduction}

Let $G$ be a finitely generated countable group. Let $\mu$ be a 
probability measure on $G$, i.e. a function $\mu : G \to \mathbb R^+$ 
such that $\sum_{g \in G} \mu(g) =1$.
Let $
S_\mu :=\{g \in G \hbox{ : } \mu(g)>0\}$, the support of $\mu$.
We say that $\mu$ is \emph{non-degenerate} if $S_\mu$ generates $G$ as a semigroup.
(We do not require $S_\mu$ to be finite.)

Let $|\cdot|$ be a word metric on $G$ associated to some finite generating set. 
(We do not assume any connection between this set and $S_\mu$.) We say that $\mu$ has finite first moment if 
\[
\sum_{g \in G} |g| \mu(g) <\infty
\]
and that $\mu$ has finite exponential moment of order $c>0$ if
\[
\sum_{g \in G} e^{c|g|} \mu(g) <\infty.
\]

The measure $\mu$ defines a random walk on $G$ 
with transition probabilities
$
p(s,t)
= \mu(s^{-1}t)$.
The convolution $\mu \ast \nu$ of two functions $\mu,\nu : G \to \mathbb R^+$ is defined 
by
\[
\mu \ast \nu(g) = \sum_{s \in G} \mu(s) \nu(s^{-1}g).
\]
We will be interested in the {\it spectral radius} of this random walk, defined  by
$$
\lambda(G,\mu) :=\limsup_{n\to\infty}(\mu^{\ast n}(e))^{1/n}.
$$

Clearly, $0 \le \lambda(G,\mu) \le 1$.
A celebrated theorem of Kesten (which does not even require $G$ to be
finitely generated) says that if $\mu$ is symmetric then
$\lambda(G,\mu)=1$ if and only if $G$ is amenable \cite{Kesten}. 
(We recall that $G$ is amenable if and only if it admits a Banach mean, i.e.
a linear functional $M:\ell^\infty(G)\to\mathbb{R}$ such that $M(\mathds{1})=1$, $\inf_{g\in G} f(g) \le M(f)\le \sup_{g\in G} f(g)$, and $M(f_g)=M(f)$,
where $f_g(x)=f(gx)$. See the papers of F{\o}lner \cite{folner} and Day \cite{Day} for further discussion.)
The aim of this note is to generalise Kesten's criterion to the
non-symmetric case.

To state our generalisation, we need to consider the abelianisation of $G$.
Since $G$ is finitely generated, this has a finite rank $k \ge 0$.
Let $G^{\mathrm{ab}} = G/[G,G]$ denote the
abelianisation of $G$ and let $G^{\mathrm{ab}}_{\mathrm{T}}$ denote the torsion subgroup
of $G^{\mathrm{ab}}$. Now set $\overline{G} = G^{\mathrm{ab}}/G^{\mathrm{ab}}_{\mathrm{T}}
\cong \mathbb Z^k$, for some $k \ge 0$,
(the torsion-free part of the abelianisation) and
let $\pi : G \to \overline G$ be the natural 
projection homomorphism.
Write $\bar{\mu} =\pi_*(\mu)$, i.e.
\[
 \bar{\mu}(m)=\sum_{\substack{g\in G \\ \pi(g)=m }}\mu(g).
 \]

\begin{theorem}[Non-symmetric Kesten criterion]\label{thm:kesten}
Let $G$ be a finitely generated group and let $\mu$ be a non-degenerate
probability measure on $G$.
Then
$$
G \; \mathrm{amenable} \; \iff \lambda(G,\mu) = \lambda(\overline{G},\bar{\mu}).
$$
\end{theorem}

The special case where $\mu$ has finite support originally appeared in Dougall--Sharp \cite{DougallSharp},
where it is written in the language of subshifts of finite type and Gibbs measures. 

The value of $\lambda(\overline{G},\bar{\mu})$ may be characterised in the following way.
Define $\phi_\mu : \R^k \to \R^+ \cup \{+\infty\}$ by
\[
\phi_\mu(v) = \sum_{g \in G} e^{\langle v,\pi(g)\rangle} \mu(g) 
=
\sum_{m \in \Z^k}  e^{\langle v,m\rangle} \bar{\mu}(m),
\]
where $\langle v,m\rangle = v_1m_1 + \cdots + v_km_k$.
By a result of Stone \cite{Stone},\cite{Stone_Berkeley}, there is a unique $\xi \in \R^k$
at which $\phi_\mu(v)$ attains its minimum. Then $\lambda(\overline{G},\bar{\mu}) = \phi_\mu(\xi)$.
This is discussed in more detail in Section \ref{sec:limit_theorems}.

A probability measure $\mu$ (with finite first moment) is said to be \emph{centred} if for each homomorphism 
 $\chi : G \to \R$, we have
 \[
 \sum_{g \in G}  \chi(g) \mu(g) =0.
 \]
 Any such homomorphism factors through $\overline{G}$ so it is easy to see that
$\mu$ is centred if and only if either $k=0$ or
\[\sum_{g \in G} \pi(g) \mu(x) = \sum_{m \in \bar G} m\bar \mu(m)=0.
\]
In particular, $\mu$ is centred if and only if $\bar{\mu}$ is centred.

If, in addition, $\mu$ has a finite exponential moment of some order then we have the following result.

\begin{corollary}\label{cor:centred}
Let $G$ be a finitely group and let $\mu$ be a non-degenerate
probability measure on $G$. 
Provided $\mu$ has a finite exponential moment of some order,
we have $\lambda(G,\mu)=1$ if and only if $G$ is amenable and $\mu$ is centred.
\end{corollary}

Theorem \ref{thm:kesten} allows us to prove a ratio limit theorem for amenable groups with an
explicit limit.
To avoid any parity issues, it is convenient to restrict to aperiodic walks. We say that 
$(G,\mu)$ is {\it aperiodic}
if there exists $n_0 \ge 1$ such that $\mu^{\ast n}(e) >0$ for all $n \ge n_0$.

\begin{theorem}[Ratio limit theorem]\label{thm:ratio}
Suppose that $G$ is a finitely generated amenable and that $\mu$ is a 
non-degenerate probability measure on $G$.  Assume in addition that $(G,\mu)$ is aperiodic.
 Then, for each $g\in G$,
$$
\lim_{n \to \infty} \frac{\mu^{\ast n}(g)}{\mu^{\ast n}(e)} = e^{-\langle \xi,\pi(g)\rangle},
$$
where $\xi \in \R^k$ is the unique value for which $\lambda(G,\mu)=\phi_\mu(\xi)$.
\end{theorem}

\begin{remark}
One should compare this with a theorem of Avez \cite{Avez} that says that if $G$ is amenable and 
$\mu$ is \emph{symmetric}, non-degenerate and aperiodic then
$\lim_{n \to \infty} \mu^{\ast n}(g)/\mu^{\ast n}(e) =1$, for all $g \in G$. 
\end{remark}

\begin{remark}
It should be noted that there is no {\it a priori} mechanism to guarantee that the ratios do indeed have a limit. 
However, notice that if one has the ratio limits 
\[
\lim_{n\to\infty}\frac{\mu^{\ast n}(g)}{\mu^{\ast n}(e)} = e^{-\langle \xi, \pi(g)\rangle},
\]
for all $g\in G$, for \emph{some} $\xi$
then $G$ is necessarily amenable. 
We give the short demonstration. From the hypothesis we have for any $s\in G$,
$$
\frac{\mu^{\ast n}(s^{-1})}{\mu^{\ast n}(e)} = e^{\langle \xi, \pi(s)\rangle}.
$$
Now, since
$$
\mu^{\ast(n+1)}(g) = \sum_{s\in G} \mu(s) \mu^{\ast n}(s^{-1} g),
$$
we then have
\begin{equation}\label{eq:limit_is_1}
\lim_{n\to\infty}\frac{\mu^{\ast (n+1)}(e)}{ \mu^{\ast n}(e)} =  \lim_{n \to \infty} \sum_{s\in S_\mu} \mu(s) 
 \frac{\mu^{\ast n}(s^{-1})}{\mu^{\ast n}(e)} \ge \sum_{s\in S_\mu} \mu(s) e^{\langle \xi, \pi(g)\rangle} = \phi_\mu(\xi).
\end{equation}
In particular $\phi_\mu(\xi)<\infty$. We proceed with the proof assuming that $\xi=0$ and deduce the general case after.

Now using that
$$
\frac{\mu^{\ast n}(e)}{\mu^{\ast 1}(e)} = \prod_{m=2}^n \frac{\mu^{\ast m}(e)}{\mu^{\ast (m-1)}(e)} 
$$ 
we see that (\ref{eq:limit_is_1}) with $\phi_\mu(0)=1$ implies that $\limsup_{n\to\infty}(\mu^{\ast n}(e))^{1/n}=1$. 
This contradicts Day's \cite{Day} generalisation of Kesten's criterion to the random walk operator spectral radius --- a consequence of which is that, for any non-degenerate probability, we have that if $G$ is non-amenable then $\limsup_{n\to\infty}(\mu^{\ast n}(e))^{1/n} <1$. 

For the general case $\xi\ne0$, we have already shown that $\phi_\mu(\xi)<\infty$, and so
$$
\hat{\mu}(g) = \frac{e^{\langle \xi,\pi(g) \rangle}}{\phi_\mu(\xi)} \mu(g)
$$
is a well-defined probability measure on $G$ with ratio limits equal to one, and we again conclude that $G$ is amenable.

\end{remark}

Let us now outline the contents of the rest of the paper.
In Section \ref{sec:limit_theorems}, we recall results of Stone on random walks on $\Z^k$
that are essential to the formulation of our results, and the rather general results of Gerl.
In Section \ref{sec:proof_of_centred}, we give a proof of Corollary \ref{cor:centred} assuming
Theorem \ref{thm:kesten}.
We prove Theorem \ref{thm:kesten}
in Sections \ref{section:toamenable} and \ref{section:amenableto}.
Theorem \ref{thm:ratio} is proved in Section \ref{section:equidistribution}, as a consequence of equidistribution results for countable state shifts.

\subsection*{Acknowledgements}
We are grateful to Wolfgang Woess for bringing to our attention the work of Gerl. We are grateful to R\'emi Coulon for asking whether Theorem \ref{thm:ratio} had a converse.

\section{Results of Stone and Gerl
}
\label{sec:limit_theorems}

In this section, we recall classic results of Stone concerning random walks on $\Z^k$. 
Let $\omega$ be a non-degenerate aperiodic probability measure on $\Z^k$
and define $\phi_\omega : \R^k \to \R \cup \{+\infty\}$ by
\[
\phi_\omega(v) 
= \sum_{m \in \Z^k} e^{\langle v,m \rangle} \omega(m).
\]

\begin{lemma}[Stone \cite{Stone}, \cite{Stone_Berkeley}]\label{lem:stone}
If $\omega$ is non-degenerate then there is a unique $\xi \in \R^k$ such that
$\phi_\omega(\xi) = \inf_{v \in \R^k} \phi_\omega(v)$.
Furthermore, $\lambda(\Z^k,\omega)=1$ if and only if $\phi_\omega(\xi) = 1$, i.e. if and only if $\xi=0$. 
\end{lemma}

We note that $\phi_\omega(\xi)=1$ if and only if $\phi_\omega(v) \ge 1$ for all $v \in \R^k$.

\begin{corollary}\label{cor:xi_measure}
$\lambda(\Z^k,\omega) = \phi_\omega(\xi)$.
\end{corollary}

\begin{proof}
Suppose that $\phi_\omega(\xi)<1$. Then we can define a
new probability measure $\omega_\xi$ on $\Z^k$ by
\[
\omega_\xi(m) = (\phi_\omega(\xi))^{-1}  e^{\langle \xi,m \rangle} \omega(m).
\]
Then $\omega_\xi$ has the same support as $\omega$ and 
\[
\omega_\xi^{\ast n}(m) = (\phi_\omega(\xi))^{-n} e^{\langle \xi,m \rangle} \omega^{\ast n}(m).
\]
We have
\[
\sum_{m \in \Z^k} e^{\langle v,m \rangle } \omega_\xi(m) 
=\frac{1}{\phi_\omega(\xi)} \sum_{m \in \Z^k} e^{\langle v+\xi,m \rangle } \omega(m)
= \frac{\phi_\omega(v+\xi)}{\phi_\omega(\xi)}
\ge 1.
\]
Hence, $\lambda(\Z^k,\omega_\xi)=1$ and
\[
\lambda(\Z^k,\omega) = \phi_\omega(\xi) \lambda(\Z^k,\omega_\xi) = \phi_\omega(\xi).
\]
\end{proof}

We now state a ratio limit theorem due to Stone.

\begin{proposition}[Stone \cite{Stone}]\label{Prop:Stone}
Suppose that $\omega$ is non-degenerate and aperiodic. Then, for each $m \in \mathbb Z^k$,
\begin{equation}
\lim_{n \to \infty} \frac{\omega^{*n}(m)}{\omega^{*n}(0)} = e^{-\langle \xi,m \rangle}.
\end{equation}
\end{proposition}

Ratio limit theorems are intimately related to the existence of harmonic functions.
Given a random walk $(G,\mu)$,
we define the random walk operator $P_\mu : \ell_\mu^1(G) \to \ell_\mu^1(G)$ by
\[
P_\mu f(g) = \sum_{s \in G} p(g,s) f(s)= \sum_{s \in G} \mu(g^{-1}s) f(s).
\]
This may be written as a convolution $P_\mu f = f \ast \check{\mu}$, 
where $\check{\mu}(s)=\mu(s^{-1})$.
A function
$f : G \to \mathbb R^+$ is called $\lambda$-\emph{harmonic} if $P_\mu f=\lambda f$,
 i.e. if 
 \[
 \sum_{s \in G} \mu(g^{-1}s)f(s) = \lambda f(g).
 \]
 (Some authors define $f$ to be $\lambda$-harmonic if $\mu \ast f = \lambda f$.)
 If we write $h_\xi(m) = e^{-\langle \xi,m\rangle}$ for the limit in Proposition 
 \ref{Prop:Stone} then we see that the function $\check{h}_\xi(m) := h_\xi(-m) = e^{\langle \xi,m \rangle}$
 is $\lambda(\Z^k,\omega)$-harmonic (for $\omega$).
 Furthermore $\lambda=\lambda(\Z^k,\omega)$ is the smallest value for which there
 is a $\lambda$-harmonic function.

One may ask about ratio  limit theorems on more general groups than $\Z^k$.
Following earlier work by Avez \cite{Avez} and Gerl \cite{Gerl73},
a rather general ratio limit theorem was proved by Gerl \cite{Gerl78}, where it is obtained as a 
 corollary 
of the following limit theorem.
A detailed account and discussion may be found in the recent note by Woess \cite{Woess22}.

\begin{proposition}[Gerl's fundamental theorem \cite{Gerl78}]\label{thm:slowincrease}
Suppose that $\mu$ is a 
non-degenerate probability measure on $G$ such that
$(G,\mu)$ is aperiodic. 
Then we have
$$
\lim_{n \to \infty}
\frac{\mu^{\ast (n+1)}(e)}{\mu^{\ast n}(e)} = \lambda(G,\mu).
$$
\end{proposition}

Gerl used this to prove the following conditional ratio limit theorem.

\begin{proposition}[Gerl's ratio limit theorem \cite{Gerl78}] \label{prop:gerl_ratio}
Suppose that $\mu$ is a 
non-degenerate probability measure on $G$ such that $(G,\mu)$ is aperiodic.
Suppose there is a set $\mathfrak F \subset \{f : G \to \mathbb R^+ \hbox{ : } f(e)=1\}$
such that
\begin{enumerate}
\item[(1)]
if $f : G \to \mathbb R^+$ is defined by
\[
f(g) = \lim_{j \to \infty} \frac{\mu^{\ast n_j}(g)}{\mu^{\ast n_j}(e)},
\]
for some subsequence $(n_j)_{j=1}^\infty$, then $f \in \mathfrak F$;
\item[(2)]
there exists a unique $h \in \mathfrak F$ satisfying
the equation $\mu \ast h = \lambda(G,\mu) h$.
\end{enumerate}
Then 
\[
\lim_{n \to \infty} \frac{\mu^{\ast n}(g)}{\mu^{\ast n}(e)} = h(g),
\]
for all $g \in G$.
\end{proposition}

In particular, if we have uniqueness of a $\lambda(G,\mu)$-harmonic function for $(G,\mu)$ then we know the ratio limit theorem holds.
The advantage of our Theorem \ref{thm:ratio}, for amenable groups, is that we don't consider arbitrary harmonic functions instead we directly work with functions coming from the abelianisation.

\section{Proof of Corollary \ref{cor:centred}}\label{sec:proof_of_centred}

In this section we prove Corollary \ref{cor:centred}, assuming Theorem \ref{thm:kesten}.
We note that $\phi_\mu = \phi_{\bar{\mu}}$, so we can use Lemma \ref{lem:stone}.

\begin{proof}[Proof of Corollary \ref{cor:centred}]
If $G$ is not amenable then 
Theorem \ref{thm:kesten} tells us that $\lambda(G,\mu) < \lambda(\overline{G},\bar{\mu}) \le 1$;
so it suffices to show that, if $G$ is amenable, then $\lambda(\overline{G},\bar{\mu}) =1$ if and only if $\mu$ is centred. 

Since $\bar{\mu}$ is non-degenerate, $\phi_\mu(v)$ is strictly convex on the set where
it is finite.
The hypothesis of a finite exponential moment implies that $\phi_\mu(v)$ is finite
and differentiable in some neighbourhood of $0 \in \R^k$. 
Therefore, $\phi_\mu(v)$ has its unique minimum at $v=0$ if and only if $\nabla \phi_\mu(0)=0$.
Suppose that
$\lambda(\overline{G},\bar{\mu})=1$; then, by Lemma \ref{lem:stone}, $\phi_\mu$ has its minimum at $0$
and so
\[
0 = \nabla \phi_\mu(0) = \sum_{m \in \Z^k} m \bar{\mu}(m),
\]
i.e. $\mu$ is centred. On the other hand, if $\lambda(\overline{G},\bar{\mu})<1$ then,
again by Lemma \ref{lem:stone}, the unique minimum of $\phi_\mu$ is not at $0$ and so
\[
0 \ne \nabla \phi_\mu(0) = \sum_{m \in \Z^k} m \bar{\mu}(m),
\]
i.e. $\mu$ is not centred.
\end{proof}

\section{Proof of Theorem \ref{thm:kesten} ($\impliedby$)}\label{section:toamenable}

In this section we show that if $\lambda(G,\mu)=\lambda(\overline G,\bar \mu)$ then $G$ is amenable. (In Kesten's original theorem, this was the harder direction but in our case
it is the easier implication.) Noting that $\phi_\mu = \phi_{\bar{\mu}}$, recall from Section \ref{sec:limit_theorems} that there is a unique
$\xi \in \R^k$ such that
\[
\phi_\mu(\xi) = \inf_{v \in \R^k} \phi_\mu(v)
\]
and 
$$
\phi_\mu(\xi) = \sum_{g\in G}  \mu(g)e^{\langle \xi,\pi(g)\rangle} =\lambda(\overline{G},\bar{\mu}).
$$
We
define a new probability measure $\mu_\xi$ on $G$ (analogous to the measure $\omega_\xi$ on
$\Z^k$ in the proof of Corollary \ref{cor:xi_measure}) by
\[
\mu_\xi(g) =\phi_\mu(\xi)^{-1}  e^{\langle \xi,\pi(g)\rangle} \mu(g).
\]

\begin{proof}[Proof of Theorem \ref{thm:kesten} ($\impliedby$)]
Assume that $G$ is non-amenable. 
By Theorem 1 of Day \cite{Day} (see also Theorem 5.4 of Stadlbauer \cite{Stadlbauer1}), 
we see that the probability measure $\mu_\xi$ satisfies
$$
\limsup_{n\to\infty} (\mu_\xi^{\ast n}(e))^{1/n}<1.
$$
Unpicking the definitions, $\mu_\xi^{\ast n}(e) = \phi_\mu(\xi)^{-n} \mu^{\ast n}(e)$. Hence
$
\limsup_{n\to\infty} (\mu^{\ast n}(e))^{1/n}<\phi_\mu(\xi).
$
\end{proof}

\section{Proof of Theorem \ref{thm:kesten} ($\implies$)}\label{section:amenableto}

In this section we will show the harder implication that if $G$ is amenable 
then  $\lambda(\overline G,\bar \mu) \le \lambda(G,\mu)$, and hence that
$\lambda(\overline G,\bar \mu) = \lambda(G,\mu)$.
We remark that the proof given here is significantly easier and more direct than the one 
we gave 
in \cite{DougallSharp}.
Following that proof would introduce a family of measures, indexed by $g \in G$,
on the space $S_\mu^{\mathbb{N}} \times G$, each describing the paths that visit 
$S_\mu^{\N} \times \{g\}$. 
These measures (which are also analysed in \cite{Stadlbauer2}) are not required
here.

Let us begin by emphasising the following: though we know that $\bar{\mu}$ has a  
$\lambda(\overline{G},\bar{\mu})$-harmonic function \emph{it plays no role in this part of proof}! 
The first element of the proof is the following proposition. Subsequently, the rest of the section will 
be devoted to showing that
show its hypothesis is satisfied with $\lambda = \lambda(G,\mu)$.

\begin{proposition}\label{prop:conditional}
If there is a homomorphism $h : \overline{G} \to \mathbb{R}^{>0}$,
the multiplicative group of positive real numbers,
such that, for all $n \in \N$,
$$
\sum_{g\in G} \mu^{\ast n}(g) h(-\pi(g)) \le \lambda^n
$$
then $\lambda(\overline G,\bar{\mu})\le \lambda$.
\end{proposition}
\begin{proof}
Suppose such a homomorphism $h$ exists.
Since $h$ is positive so we can throw away the terms where $-\pi(g)\ne 0$
and obtain
$$
\sum_{\substack{g\in G \\ \pi(g)=0}} \mu^{\ast n}(g) \le \lambda^n.
$$
Hence, for any $\delta<1$,
$$
\sum_{n\in\mathbb{N}}
\bar{\mu}^{\ast n}(0)(\delta^{-1}\lambda)^{-n}
=\sum_{n\in\mathbb{N}}
\sum_{\substack{g\in G \\ \pi(g)=0}} \mu^{\ast n}(g)\delta^{n}\lambda^{-n} <\infty. 
$$
This says that $\lambda(\overline{G},\bar{\mu}) \le \lambda\delta^{-1}$. Since we can take $\delta$ arbitrarily close to $1$ we are done.
\end{proof}

We view $\lambda(G,\mu)$ as a convergence parameter for the series 
$$
\zeta(t)= \sum_{n\in\mathbb{N}} \mu^{\ast n}(e) t^{-n},
$$ 
where $t>0$,
i.e.
$$
\inf\left \{ t\in \mathbb{R}^+ : \sum_{n\in\mathbb{N}} \mu^{\ast n}(e) t^{-n}<\infty\right\} = \limsup_{n\to\infty}(\mu^{\ast n}(e))^{1/n}
= \lambda(G,\mu).
$$
This formulation is reminiscent of the Poincar\'e series used in the construction of Patterson--Sullivan measures on the limit sets of Kleinian groups and more general limit sets and, indeed,
we employ ideas from this theory
The most relevant reference here is Roblin \cite{Roblin},
which covers the basic tools of Patterson-Sullivan theory and the amenability ``trick'' we 
will use in the proof of Proposition \ref{prop:h_exists}.

The series $\zeta(t)$ does not necessarily diverge at $t=\lambda(G,\mu)$
but one can modify the series,
in a controlled way, to guarantee divergence at this critical parameter. 
The following appears as Lemma 3.2 in Denker and Urbanski \cite{urbanski}
(see also \cite{Stadlbauer2}).

\begin{lemma}\label{lem:slowly_diverging}
Let $(a_n)_{n=1}^\infty$ be a sequence in $\R^+$ and let $\rho = \limsup_{n\to\infty} a_n^{1/n}$.
Then there is a 
sequence $(b_n)_{n=1}^\infty$ of positive real numbers such that
$\lim_{n \to \infty} b_{n+1}/b_n=1$ 
for which $\sum_{n\in\mathbb{N}} a_nb_n t^{-n} < \infty$ for $t >\rho$ but
$$
\sum_{n\in\mathbb{N}} a_nb_n\rho^{-n}=\infty.
$$
\end{lemma}

Let $(b_n)_{n=1}^\infty$ be the 
sequence given by Lemma \ref{lem:slowly_diverging}
 for the series with $a_n = \mu^{\ast n}(e)$. We prefer to use $c_n=1/b_n$,
 a decreasing sequence.
 Note that we have
 $\lim_{n \to \infty} c_{n-r}/c_n =1$ for all $r \in \mathbb N$.
We will work with a modified series $\zeta_c^e(t)$ defined by
$$
\zeta^{e}_c(t) = \sum_{n\in\mathbb{N}} \frac{\mu^{\ast n}(e)}{c_n} t^{-n},
$$
and also, for each $g\in G$, the series
$$
\zeta^{g}_c(t) = \sum_{n\in\mathbb{N}} \frac{\mu^{\ast n}(g)}{c_n} t^{-n}.
$$

\begin{lemma}
For each $g\in G$,
$$
0<\inf_{\lambda(G,\mu) <t \le 2}\frac{\zeta_c^g(t)}{\zeta_c^e(t)}
\le \sup_{\lambda(G,\mu) <t \le 2}\frac{\zeta_c^g(t)}{\zeta_c^e(t)}<\infty.
$$
\end{lemma}

\begin{proof}
We begin by observing that, for every $g,h\in G$, we have
$$
\mu^{\ast (n+k)}(g)\ge  \mu^{\ast k}(gh^{-1}) \mu^{\ast n}(h)
$$
and we may choose $k \ge 1$ for which $ \mu^{\ast k}(gh^{-1}) >0$.
This gives us the inequality
\begin{align*}
\zeta_c^g(t) = \sum_{m\le k} \frac{\mu^{\ast m}(g)}{c_{m}} + \sum_{n\in\mathbb{N}} \frac{\mu^{\ast (n+k)}(g)}{c_{n+k}}
\ge 
 \sum_{m\le k} \frac{\mu^{\ast m}(g)}{c_{m}} + \sum_{n\in\mathbb{N}} \mu^{\ast k}(gh^{-1}) \frac{\mu^{\ast n}(h)}{c_{n}}\frac{c_{n}}{c_{n+k}}.
 \end{align*}
Since the numbers $c_n$ are positive and, for a fixed $k$, $\lim_{n \to \infty} c_{n}/c_{n+k}=1$, 
we have $\inf_n c_{n}/c_{n+k}>0$. Hence
$$
\zeta_c^g(t) \ge  C_1(g,k) +  \frac{C_2(g,h,k)}{C_3(k)}\sum_{n\in\mathbb{N}}  \frac{\mu^{\ast n}(h)}{c_{n}},
$$
for positive $C_1$, $C_2$ and $C_3$.
We conclude
$$
0<\inf_{\lambda(G,\mu) <t \le 2}\frac{\zeta_c^g(t)}{\zeta_c^h(t)}.
$$
Since $g,h$ are arbitrary the lemma follows.
\end{proof}

By the previous lemma
and a standard diagonal argument, there exists a sequence $t_n\to \lambda(G,\mu)$ for which the following limits are well-defined
$$
H(g) = \lim_{n\to\infty}\frac{\zeta_c^g(t_n)}{\zeta_c^e(t_n)} \in (0,\infty),
$$
for all $g \in G$.
A crucial observation is the following.

\begin{lemma}\label{lemma:main}
For any $r\in\mathbb{N}$, we have
\[
\sum_{s\in G} \mu^{\ast r}(s) H(s^{-1}g) = \lambda^r H(g)
\]
with $\lambda = \lambda(G,\mu)$.
\end{lemma}

\begin{proof}
Fix $r \in \N$ and let $\epsilon>0$. 
Since $\lim_{n \to \infty} c_{n-r}/c_n=1$, we can
choose $n_0$ such that 
$1-\epsilon \le c_{n-r}/c_n \le  1+\epsilon$,
for all $n > n_0$. 
We will also use that 
\[
\mu^{\ast n}(g)=\sum_{s\in G}\mu^{\ast r}(s)\mu^{\ast (n-r)}(s^{-1}g).
\] 
Then, for $n> n_0$,
\[
\frac{1-\epsilon}{c_{n-r}}
\sum_{s\in G}\mu^{\ast r}(s)
\mu^{\ast (n-r)}(s^{-1}g) 
\le
\frac{\mu^{\ast n}(g)}{c_n}
\le \frac{(1+\epsilon)}{c_{n-r}}
\sum_{s\in G}\mu^{\ast r}(s)\mu^{\ast (n-r)}(s^{-1}g).
\]

Setting
\[
C_1(g,t,n_0) =  \sum_{n\le n_0} \frac{\mu^{\ast n}(g)}{c_n} t^{-n},
\]
we have
\begin{align*}
t^r \sum_{n\in\mathbb{N}} \frac{\mu^{\ast n}(g)}{c_n} t^{-n} 
&\le C_1(g,t,n_0) + t^r
(1+\epsilon)
\sum_{s\in G}\mu^{\ast r}(s)
\sum_{n>n_0}t^{-n} \frac{\mu^{\ast (n-r)}(s^{-1}g)}{c_{n-r}}
\\
&\le C_1(g,t,n_0) + (1+\epsilon)\sum_{s\in G}\mu^{\ast r}(s)H(s^{-1}g)\zeta^e_c(t),
\end{align*}
using that the terms in the series are non-negative. This gives
$$
\lambda^r H(g) \le  
\lim_{m\to\infty} \frac{C_1(g,t_m,n_0)}{\zeta^e_c(t_m)} 
+ (1+\epsilon)\sum_{s\in G}\mu^{\ast r}(s)H(s^{-1}g)=(1+\epsilon)\sum_{s\in G}\mu^{\ast r}(s)H(s^{-1}g)
$$
and, since $\epsilon$ is arbitrary, 
\[
\lambda^r H(g) \le  
\sum_{s\in G}\mu^{\ast r}(s)H(s^{-1}g).
\]

For the lower bound, we have
\begin{align*}
t^{r} \sum_{n\in\mathbb{N}} \frac{\mu^{\ast n}(g)}{c_n} t^{-n} 
&\ge  
(1-\epsilon)
\sum_{s\in G}\mu^{\ast r}(s)\sum_{n>n_0}t^{-(n-r)}\frac{\mu^{\ast (n-r)}(s^{-1}g)}{c_{n-r}} 
\\
&=
(1-\epsilon)
\sum_{s\in G}\mu^{\ast r}(s)\sum_{n \in \N}t^{-n}\frac{\mu^{\ast n}(s^{-1}g)}{c_{n}}.
\end{align*}
This gives
\begin{align*}
\lambda^{r} H(g) &\ge  
\sum_{s\in G}\mu^{\ast r}(s)H(s^{-1}g). 
\end{align*}
\end{proof}

Lemma \ref{lemma:main} gives us the following corollary.

\begin{corollary}\label{cor:cor_to_main_lemma}
For any fixed $\gamma \in G$, we have
$$
0< \inf_{g\in G}\frac{H(\gamma g)}{H(g)} \le
\sup_{g\in G}\frac{H(\gamma g)}{H(g)}<\infty.
$$
\end{corollary}

\begin{proof}
Given $\gamma \in G$, we can find $x_1,\ldots,x_k \in S_\mu$, for some $k \ge 1$,
such that
$x_1\cdots x_k=\gamma^{-1}$. This gives us
$$ 
\mu(x_1)\cdots \mu(x_k)H(\gamma g)
\le 
\sum_{(s_1,\ldots,s_k)\in G^k} \mu(s_1)\cdots \mu(s_k) H((s_1 \cdots s_k)^{-1} g) = \lambda^k H(g),
$$
and so $\sup_{g \in G} H(\gamma g)/H(g)$ is finite.

Now we put $\gamma g$ on the right hand side and choose 
$y_1,\ldots,y_\ell \in S_\mu$ such that
$y_1 \cdots y_\ell=\gamma$ to get 
$$
\mu(y_1)\cdots \mu(y_\ell)H(g)
\le
\sum_{(s_1,\ldots,s_\ell)\in G^\ell}
\mu(s_1) \cdots \mu(s_\ell) H((s_1 \cdots s_\ell)^{-1}\gamma g) 
= \lambda^\ell H(\gamma g),
$$
and so $\inf_{g \in G} H(\gamma g)/H(g)$ is positive.
\end{proof}

We are now ready to use the amenability of $G$. 
We use the existence of a Banach mean on $\ell^\infty(G)$ to 
``average over $g$" in the last lemma. 

\begin{proposition}\label{prop:h_exists}
There is a homomorphism $h:\overline{G}\to \mathbb{R}^{>0}$ such that, for all $n \in \N$,
$$
\sum_{s\in G} \mu^{\ast n}(s)  h(-\pi(s)) \le \lambda^n,
$$
with $\lambda = \lambda(G,\mu)$.
\end{proposition}

\begin{proof}
Since any homomorphism $h : G \to \R^{>0}$ factors through $\overline G$, it suffices to show that there is a homomorphism $h : G \to \R^{>0}$ such that, for all $n \in \N$, we have
\begin{equation}\label{eq:bound_for_h}
\sum_{s\in G} \mu^{\ast n}(s)  h(s^{-1}) \le \lambda^n.
\end{equation}
Let $M$ be a Banach mean on $\ell^\infty(G)$.
Jensen's inequality says that if $\varphi$ is convex then
$$
M(\varphi( f))\ge \varphi (M(f)).
$$
(This is more familiar when the linear functional is integration with respect to a probability measure. One can check that we only need monotonicity, finite additivity, and the unit normalisation $M(1)=1$.)
We apply this to the function
$g \mapsto  (H(\gamma g)/H(g))$ in $\ell^\infty(G)$. (Note we have used
Corollary \ref{cor:cor_to_main_lemma}
to know that $g \mapsto \log (H(\gamma g)/H(g))$ is in $\ell^\infty(G)$.)
Thus we obtain
\begin{align*}
M\left( g\mapsto \frac{H(\gamma g)}{H(g)}\right) &= M\left( g\mapsto \exp \log \frac{H(\gamma g)}{H(g)}\right) 
\\
&\ge \exp M\left( g\mapsto  \log \frac{H(\gamma g)}{H(g)}\right).
\end{align*}

Now set 
\[
h(\gamma) = \exp M\left( g\mapsto \log \frac{H(\gamma g)}{H(g)}\right).
\]
We will show that $h$ satisfies (\ref{eq:bound_for_h}), recalling that $M$ is only finitely
additive.
Let $\{g_k\}_{k=1}^\infty$ be an enumeration of $G$ and, for $N \ge 1$,
let $G_N = \{g_1,\ldots,g_N\}$.
Lemma \ref{lemma:main} gives us that, for any $n \ge 1$ and any $N \ge 1$,
we have
\begin{align*}
\lambda^n &= M\left( g\mapsto \sum_{s \in G} \mu^{\ast n}(s) \frac{H(s^{-1}g)}{H(g)}\right)
\ge
M\left( g\mapsto \sum_{s \in G_N} \mu^{\ast n}(s) \frac{H(s^{-1}g)}{H(g)}\right)
\\
&= \sum_{s\in G_N} \mu^{\ast n}(s) M\left( g\mapsto \frac{H(s^{-1}g)}{H(g)}\right)
\ge \sum_{s\in G_N} \mu^{\ast n}(s) \exp M\left( g\mapsto \log \frac{H(s^{-1}g)}{H(g)}\right)
\\
&=\sum_{s\in G_N} \mu^{\ast n}(s) h(s^{-1}).
\end{align*}
Taking the supremum over $N$ gives (\ref{eq:bound_for_h}).

It remains to show that $h$ is a homomorphism. Notice that 
\begin{align*}
\log h(ab)&= M\left( g\mapsto \log \frac{H(abg)}{H(g)}\right) 
\\
&=  M\left( g\mapsto \log \frac{H(abg)}{H(bg)}\right)+  M\left( g\mapsto \log \frac{H(bg)}{H(g)}\right) 
\\
&= \log h(a) + \log h(b)
\end{align*}
and
\begin{align*}
\log h(\gamma^{-1}) &= M\left(g \mapsto \log \frac{H(\gamma^{-1}g)}{H(g)}\right)
=M\left(g \mapsto \log \frac{H(g)}{H(\gamma g)}\right)
\\
&= M\left(g \mapsto -\log \frac{H(\gamma g)}{H(g)}\right) = -\log h(\gamma).
\end{align*}
using translation invariance of $M$. The conclusion follows.
\end{proof}

Combining Proposition \ref{prop:conditional} and Proposition \ref{prop:h_exists}
shows that if $G$ is amenable then $\lambda(\overline{G},\bar{\mu}) \le \lambda(G,\mu)$.

\section{Equidistribution and proof of the ratio limit theorem}\label{section:equidistribution}

In this section we use Theorem \ref{thm:kesten} to 
prove a ratio limit theorem for amenable groups, Theorem \ref{thm:ratio}.
Our arguments will also give a new proof of 
Proposition \ref{thm:slowincrease} in this setting.
Our approach is based on weighted equidistribution results for countable state shift spaces.
Suppose that $G$ is amenable, that $\mu$ is non-degenerate
 and that 
$(G,\mu)$ is aperiodic.
We let $\lambda$ denote the common value
\[
\lambda = \phi_\mu(\xi) = \lambda(\overline{G},\bar{\mu}) = \lambda(G,\mu),
\]
 given by Theorem \ref{thm:kesten}. As above, $\mu_\xi(g)=\lambda^{-1}h(g) \mu(g)$,
 where $h(g)=e^{\langle \xi,\pi(g)\rangle}$.

 We consider the sequence space $\Sigma = S_\mu^{\mathbb{N}}$ 
 and let $\sigma : 
\Sigma \to \Sigma$ be the shift map: $\sigma((s_i)_{i=1}^\infty)=(s_{i+1})_{i=1}^\infty$.
 If $s = (s_1,\ldots,s_n) \in S_\mu^n$, we write $[s]= [s_1,\dots, s_n]$ for the cylinder set
 \[
 [s_1,\dots, s_n] :=\{(x_i)_{i=1}^\infty \in \Sigma \hbox{ : } x_i =s_i \ \forall i=1,\ldots,n\}.
 \]
 We give $\Sigma$ the topology generated by cylinder sets (which are both open and closed).
 

 We denote by $\nu_\xi$ the Bernoulli measure on $\Sigma$ given by 
$$
\nu_\xi([s_1,\ldots, s_n])= \mu_\xi(s_1) \cdots  \mu_\xi(s_n).
$$

Let 
\[
\Lambda_n = \{s=(s_1,\ldots,s_n) \in S_\mu^n \hbox{ : } s_1 \cdots s_n =e\}
\]
and
define a sequence of probability measures $m_n$ on $\Sigma$
 by
\begin{align*}
m_n &:= \frac{1}{\mu_\xi^{\ast n}(e)}
\sum_{s=(s_1,\ldots, s_n) \in \Lambda_n} 
\mu_\xi(s_1)\cdots \mu_\xi(s_n) 
\delta_{s_\infty}
=
\frac{1}{\mu_\xi^{\ast n}(e)}
\sum_{s=(s_1,\ldots, s_n) \in \Lambda_n} 
\nu_\xi([s]) 
\delta_{s_\infty},
\end{align*}
where we use the notation $s_\infty \in \Sigma$ to mean the one-sided infinite concatenation of 
$s=(s_1, \ldots, s_n)$ and $\delta_{s_\infty}$ denotes the Dirac measure at this point.
We remark that we also have
\[
m_n = \frac{1}{\mu^{\ast n}(e)}
\sum_{s=(s_1,\ldots, s_n) \in \Lambda_n}
\mu(s_1)\cdots \mu(s_n) 
\delta_{s_\infty}
\]
but we do not use this formula.
We will need to explicitly evaluate the measures $m_n$ on cylinder sets.

\begin{lemma}\label{lem:explicit_natural}
For a cylinder set $[u_1,\ldots,u_k]$ we have that, for $n>k$,
\[
m_n([u_1,\ldots, u_k]) =
\frac{\mu_\xi(u_{1})\cdots \mu_\xi(u_k)}{\mu_\xi^{\ast n}(e)}\mu_\xi^{\ast (n-k)}(u^{-1}),
\]
where $u = u_1 \cdots u_k$.
\end{lemma}

\begin{proof}
This is a straightforward calculation. 
For $n >k$,
\begin{align*}
m_n([u_1,\ldots, u_k]) &= \frac{1}{\mu_\xi^{\ast n}(e)}
\sum_{(s_1,\ldots, s_n) \in \Lambda_n} 
\mu_\xi(s_1)\cdots \mu_\xi(s_n) \, \delta_{(s_1, \ldots, s_n)_\infty}([u_1,\ldots, u_k])
\\
&=
 \frac{1}{\mu_\xi^{\ast n}(e)}
 \sum_{\substack{(s_1,\ldots, s_n) \in \Lambda_n \\ 
 s_1=u_1,\cdots , s_k = u_k}} 
 \mu_\xi(s_1)\cdots \mu_\xi(s_n)
\\
&=
 \frac{\mu_\xi(u_{1})\cdots \mu_\xi(u_k)}{\mu_\xi^{\ast n}(e)}
 \sum_{\substack{(s_{k+1},\ldots, s_n) \in S_\mu^{n-k} 
 \\ s_{k+1}\cdots s_n=u^{-1}}} \mu_\xi(s_{k+1})\cdots \mu_\xi(s_n)
\\
&=
 \frac{\mu_\xi(u_{1})\cdots \mu_\xi(u_k)}{\mu_\xi^{\ast n}(e)}\mu_\xi^{\ast (n-k)}(u^{-1}).
\end{align*}
\end{proof}

We will show that, for each cylinder set $[u_1,\ldots,u_k]$. 
$m_n([u_1,\ldots,u_k])$ converges to $\nu_\xi([u_1,\ldots,u_k])$, as $n \to \infty$. 
In order to have \emph{convergence} (as opposed to an accumulation point) we need 
$\mu_\xi^{\ast n}(e)^{1/n}$ to have a limit. This is a consequence of aperiodicity,
as the next lemma shows.

\begin{lemma}\label{lemma:limit}
We have
\[
\lim_{n \to \infty} (\mu_\xi^{\ast n}(e)^{1/n}) =1.
\]
\end{lemma}

\begin{proof}
We know that $\limsup_{n \to \infty} (\mu_\xi^{\ast n}(e))^{1/n}=1$.
Since $\mu$ is aperiodic, we have $\mu^{\ast n}(e)>0$ for all sufficiently large $n$.
Recall also that $\mu_\xi^{\ast (n+m)}(e)\ge \mu_\xi^{\ast n}(e)\mu_\xi^{\ast m}(e)$. 
This tells us that $-\log \mu_\xi^{\ast n}(e)$ is sub-additive. Hence by Fekete's lemma
$$
\lim_{n\to\infty} \frac{-\log \mu^{\ast n}(e)}{n} = \inf_{n \ge 1} \frac{-\log \mu^{\ast n}(e)}{n},
$$
in particular the limit exists and is equal to the limsup.
\end{proof}

In order to show the required convergence for the $m_n$, 
we introduce some ideas and terminology from thermodynamic formalism
and large deviation theory. 
A function $\varphi : \Sigma \to \R$ is called locally H\"older continuous if
\begin{equation}\label{eq:loc_Holder}
\sup_{s \in S_\mu^n} \sup_{x,y \in [s]} |\varphi(x)-\varphi(y)| \le C\theta^n,
\end{equation}
for some $C>0$ and $0<\theta<1$, for all $n \ge 1$.
For a locally H\"older continuous function
$\varphi : \Sigma \to \R$, we define the \emph{Gurevi\v{c} pressure} 
$P_{\mathrm{G}}(\varphi)$ by
\begin{equation*}\label{eq:def_of_Gpressure}
P_{\mathrm{G}}(\varphi) = \lim_{n \to \infty} \frac{1}{n} \log
\sum_{s \in S_\mu^n} \exp \sum_{j=0}^{n-1} \varphi(s_\infty) \in \R \cup \{+\infty\}.
\end{equation*}
(The original definition given by Sarig in \cite{Sarig-etds1999} is somewhat different,
and only requires, (\ref{eq:loc_Holder}) to hold for $n \ge 2$, but, by 
Corollary 1 of \cite{Sarig-pams2003}, the above formula gives the Gurevi\v{c} pressure 
in our setting.)
We now fix
\[
\varphi((s_i)_{i=1}^\infty) := \log \mu_\xi(s_1) = \log \nu([s_1]),
\]
so that, in particular, $P_{\mathrm{G}}(\varphi)=0$. Let $\chi$ be the indicator function of some cylinder. We can easily calculate from the definition that, for $t \in \R$, $P_{\mathrm{G}}(\varphi+t\chi)
\le \max\{0,|t|\}$ for all $t \in \R$. Hence, by Corollary 4 of \cite{Sarig-pams2003}, 
$t \mapsto P_{\mathrm{G}}(\varphi+t\chi)$ is real analytic for $t \in \R$ and, 
by Theorems 6.12 and 6.5 of \cite{Sarig-PSPM2015}, 
\begin{equation}\label{eq:deriv_of_pressure}
\frac{dP(\varphi+t\chi)}{dt}\Bigg|_{t=0} = \int \chi \, d\nu_\xi.
\end{equation}
(The same discussion remains true if $\chi$ is replaced with any \emph{bounded} locally H\"older
function but indicator functions of cylinders are sufficient for our purposes.)

For $s \in S_\nu^n$, let $\tau_{s,n}$ denote the orbital measure
\[
\tau_{s,n} := \frac{1}{n}\sum_{j=0}^{n-1} \delta_{\sigma^j(s_\infty)}.
\]
Following, Theorem 7.4 of \cite{Sarig-PSPM2015}, we have the following large deviations bound.

\begin{proposition}\label{prop:local_ld_bound}
Given $\epsilon>0$, there exists
$C>0$ and $\eta(\epsilon)>0$ such that
\begin{equation*}\label{eq:local_ld_bound}
\sum_{\substack{s \in S_\mu^n \\ \left|\int \chi \, d\tau_{s,n} - \int \chi d\nu_\xi\right| > \epsilon}}
\nu_\xi([s]) \le Ce^{-\eta(\epsilon)n}.
\end{equation*}
\end{proposition}

\begin{proof}
The proof is standard but we include it for completeness.
We consider $s \in S_\mu^n$ such that
$\int \chi \, d\tau_{s,n} > \int \chi \, d\nu_\xi +\epsilon$
and $\int \chi \, d\tau_{s,n} < \int \chi \, d\nu_\xi -\epsilon$ separately.
For $t>0$, we have
\begin{align*}
\sum_{\substack{s \in S_\mu^n \\ \int \chi \, d\tau_{s,n} > \int \chi d\nu_\xi + \epsilon}}
\nu_\xi([s])
&\le
\sum_{s \in S_\mu^n} e^{\varphi^n(s_\infty) +t\chi^n(s_\infty) -nt\int \chi \, d\nu_\xi -nt\epsilon},
\end{align*}
so that, 
\begin{align*}
\limsup_{n \to \infty} 
\frac{1}{n} \log
\sum_{\substack{s \in S_\mu^n \\ \int \chi \, d\tau_{s,n} > \int \chi d\nu_\xi + \epsilon}}
\nu_\xi([s])
\le P_{\mathrm{G}}(\varphi + t\chi) - t\int \chi \, d\nu_\xi -t\epsilon.
\end{align*}
Using $P_{\mathrm{G}}(\varphi)=0$ and  (\ref{eq:deriv_of_pressure}), we see that, for sufficiently small $t>0$, we have
\[
P_{\mathrm{G}}(\varphi + t\chi) - t\int \chi \, d\nu_\xi -t\epsilon <0.
\]
Similarly, for $t<0$, 
\begin{align*}
\limsup_{n \to \infty} 
\frac{1}{n} \log
\sum_{\substack{s \in S_\mu^n \\ \int \chi \, d\tau_{s,n} < \int \chi d\nu_\xi - \epsilon}}
\nu_\xi([s])
\le P_{\mathrm{G}}(\varphi + t\chi) - t\int \chi \, d\nu_\xi +t\epsilon.
\end{align*}
and, for sufficiently small $t<0$, this upper bound is negative.
Combing these estimates gives the result. 
\end{proof}

Since $\Lambda_n \subset S_\mu^n$, we have the following immediate
corollary.

\begin{corollary}\label{cor:lambda_version_of_ld}
For $\epsilon>0$, we have
\[ 
\sum_{\substack{s \in \Lambda_n \\ \left|\int \chi \, d\tau_{s,n} - \int \chi d\nu_\xi\right| > \epsilon}}
\nu_\xi([s])
\le Ce^{-\eta(\epsilon)n}.
\]
\end{corollary}

The following equidistribution result is now an easy consequence. 

\begin{proposition}\label{prop:equidistribution}
For any cylinder set
$[u_1,\ldots, u_k]$, we have that
\begin{equation*}\label{eqn:limit_for_cylinder}
\lim_{n \to \infty} m_n([u_1,\ldots, u_k]) =\mu_\xi(u_1)\cdots \mu_\xi(u_k).
\end{equation*}
\end{proposition}

\begin{proof}
Let $\chi : \Sigma \to \R$ be the indicator function of $[u_1,\ldots,u_k]$, then we need to show
\[
\lim_{n \to \infty} \int \chi \, dm_n = \int \chi \, d\nu_\xi.
\]
We have
\begin{align*}
\int \chi \, dm_n - \int \chi \, d\nu_\xi 
&= \frac{1}{\mu_\xi^{\ast n}(e)}
\sum_{\substack{s \in \Lambda_n \\ \left|\int \chi \, d\tau_{s,n} - \int \chi d\nu_\xi\right| > \epsilon}} 
\nu_\xi([s])
\int \chi \, d\tau_{s,n} 
\\
&+  \frac{1}{\mu_\xi^{\ast n}(e)}\sum_{\substack{s \in \Lambda_n \\ \left|\int \chi \, d\tau_{s,n} - \int \chi d\nu_\xi\right| \le \epsilon}}  
\nu_\xi([s])
\int \chi \, d\tau_{s,n}-\int \chi \, d\nu_\xi.
\end{align*}
By Lemma \ref{lemma:limit} and Corollary \ref{cor:lambda_version_of_ld}, the first term on the right hand side tends to zero
exponentially fast. Also,
\begin{align*}
&\left|
\frac{1}{\mu_\xi^{\ast n}(e)}
\sum_{\substack{s \in \Lambda_n \\ \left|\int \chi \, d\tau_{s,n} - \int \chi d\nu_\xi\right| \le \epsilon}} 
\nu_\xi([s])
\int \chi \, d\tau_{s,n}-\int \chi\, d\nu_\xi
\right|
\\
&\le
\epsilon + 
\frac{\left|\int \chi \, d\nu_\xi\right|}{\mu_\xi^{\ast n}(e)} 
\sum_{\substack{s \in \Lambda_n \\ \left|\int \chi \, d\tau_{s,n} - \int \chi d\nu_\xi\right| > \epsilon}} 
\nu_\xi([s])
\le \epsilon + Ce^{-\eta(\epsilon)n},
\end{align*}
which, since $\epsilon$ is arbitrary, gives the result.
\end{proof}

Combining Proposition \ref{prop:equidistribution} with Lemma \ref{lem:explicit_natural}, we see that for $u=u_1\cdots u_k$ we have
\begin{equation}\label{eqn:steady_for_muh}
\lim_{n \to \infty} \frac{\mu_\xi^{\ast (n-k)}(u^{-1})}{\mu_\xi^{\ast n}(e)}
=1.
\end{equation}

Most of the work is done. We prove Proposition \ref{thm:slowincrease} for amenable groups.

\begin{proof}[Proof of Proposition \ref{thm:slowincrease}]
It suffices to show that
\[
\lim_{n \to \infty} \frac{\mu_\xi^{\ast (n+1)}(e)}{\mu_\xi^{\ast n}(e)} = 1.
\]
If we choose $(u_1,\ldots, u_k) \in S_\mu^k$ with $u_1\cdots u_k=e$, then equation (\ref{eqn:steady_for_muh}) gives that
$
\lim_{n \to \infty} \mu_\xi^{\ast (n-k)}(e)/\mu_\xi^{\ast n}(e)=  1$,
so that $\lim_{n \to \infty} \mu_\xi^{\ast (n+k)}(e)/\mu_\xi^{\ast n}(e)=1$.
We then write
\[
\frac{\mu_\xi^{\ast (n+k)}(e)}{\mu^{\ast n}(e)}
= \frac{\mu_\xi^{\ast (n+k)}(e)}{\mu_\xi^{\ast (n+k-1)}(e)}
\cdots
\frac{\mu_\xi^{\ast (n+1)}(e)}{\mu_\xi^{\ast n}(e)}
\]
and deduce that $\lim_{n \to \infty} \mu_\xi^{\ast (n+1)}(e)/\mu_\xi^{\ast n}(e)=1$, as required.
\end{proof}

We can now establish the ratio limit theorem for amenable groups.

\begin{proof}[Proof of Theorem \ref{thm:ratio}]
Let $g \in G$ be arbitrary. Choosing $(u_1,\ldots, u_k)$ with $u_1\cdots u_k=g^{-1}$
and applying (\ref{eqn:steady_for_muh})
 gives that
 $\lim_{n \to \infty} \mu_\xi^{\ast (n-k)}(g)/\mu_\xi^{\ast n}(e) = 1$ and hence that
$$
\lim_{n \to \infty} \frac{\mu^{\ast (n-k)}(g)}{\mu^{\ast n}(e)} = \frac{\lambda^{-k}}{h(g)}.
$$
Now, applying Proposition \ref{thm:slowincrease},
$$
\frac{\mu^{\ast n}(g)}{\mu^{\ast n}(e)} = \frac{\mu^{\ast n}(g)}{\mu^{\ast (n+k)}(e)}
\frac{\mu^{\ast (n+k)}(e)}{\mu^{\ast n}(e)} \to e^{-\langle \xi,\pi(g)\rangle},
$$
as $n\to\infty$.
\end{proof}

\end{document}